\documentclass[12pt]{article}  
\usepackage{latexsym}
\usepackage{amssymb,amsmath,graphics,graphicx}
\usepackage{tikz}
\usepackage[margin=1in]{geometry}
\usepackage{natbib}

\newtheorem{theorem}{Theorem}

\newtheorem{prop}[theorem]{Proposition}

\newtheorem{cor}[theorem]{Corollary}

\newtheorem{defi}{Definition}
\let\oldmarginpar\marginpar
\renewcommand\marginpar[1]{\-\oldmarginpar[\raggedleft\footnotesize #1]%
{\raggedright\footnotesize #1}}

\newcommand{\F}{{\cal F}}

\def\vertex(#1){\put(#1){\circle*{2}}}
\def\vertexo(#1){\put(#1){\circle{2}}}
\def\vert(#1){\put(#1){\circle*{1.5}}}
\def\verto(#1){\put(#1){\circle{1.5}}}
\def\lab(#1)#2{\put(#1){\makebox(0,0)[c]{#2}}}
\newenvironment{proof}{\noindent\textbf{Proof. }}{\hfill$\square$}

\newenvironment{claimproof}{\textsc{Proof.} }{\hfill$\diamond$}

\begin{document}
%\bibliographystyle{plain}
% Set the beginning of a LaTeX document

%\begin{frontmatter}



\title{Pebbling on Directed Graphs with Fixed Diameter}


\author{
John Asplund\\
{\small Department of Technology and Mathematics,} \\
{\small Dalton State College,} \\
{\small Dalton, GA 30720, USA} \\
{\small jasplund@daltonstate.edu}\\
\\
Franklin Kenter\\
{\small Department Mathematics}\\
{\small United States Naval Academy} \\
{\small Annapolis, MD 21401, USA} \\
{\small kenter@usna.edu}\\
\\
 }

\maketitle

\begin{abstract}
Pebbling is a game played on a graph. The single player is given a graph and a configuration of pebbles and may make pebbling moves by removing 2 pebbles from one vertex and placing one at an adjacent vertex to eventually have one pebble reach a predetermined vertex.  The pebbling number,  $\pi(G)$, is the minimum number of pebbles such that regardless of their exact configuration, the player can use pebbling moves to have a pebble reach any predetermined vertex.

Previous work has related $\pi(G)$ to the diameter of $G$. Clarke, Hochberg, and Hurlbert demonstrated that every connected undirected graph on $n$ vertices with diameter 2 has $\pi(G) = n$ unless it belongs to an exceptional family of graphs, consisting of those that can be constructed in a specific manner; in which case $\pi(G) = n +1$. By generalizing a result of Chan and Godbole, Postle showed that for a graph with diameter $d$, $\pi(G) \le n 2^{\lceil \frac{d}{2} \rceil} (1+o_n(1))$.

In this article, we continue this study relating pebbling and diameter with a focus on directed graphs. This leads to some surprising results. First, we show that in an oriented directed graph $G$ (in the sense that if $i \to j$ then we cannot have $j \to i$), it is indeed the case that if $G$ has diameter 2, $\pi(G) = n$ or $n + 1$, and if $\pi(G) = n+1$, the directed graph has a very particular structure. In the case of general directed graphs (that is, if $i \to j$, we may or may not have an arc $j \to i$) with diameter 2, we show that $\pi(G)$ can be as large as $\frac32 n + 1$, and further, this bound is sharp. More generally, we show that for general directed graphs, $\pi(G) \le 2^d n / d + f(d)$ where $f(d)$ is some function of only $d$.
\end{abstract}

%\begin{keyword}
%%% keywords here, in the form: keyword \sep keyword
%directed graphs \sep pebbling \sep diameter \sep connectivity
%%% PACS codes here, in the form: \PACS code \sep code
%
%%% MSC codes here, in the form: \MSC code \sep code
%%% or \MSC[2008] code \sep code (2000 is the default)
%05C12 \sep 05C20 \sep 05C40 \sep 05C78
%\end{keyword}

%\end{frontmatter}

%{\small {\bf Keywords:} directed graphs, pebbling, diameter, connectivity} \\
%\indent {\small {\bf AMS subject classification:  05C69}}

%\section{Outline}
%\begin{enumerate}
%\item Introduction
%\item Preliminaries and Notation
%\item $n$ or $n+1$, then 3/2 Result
%\item $k(G)=2$, diameter 2
%\end{enumerate}

\section{Introduction and Preliminaries}\label{intro}

%\franklin{Moved everything to Section 1}

Pebbling on graphs is a game that was first mentioned by Lagarias and Saks (via a private
communication to Chung) in relation to a problem by \citep{LK} and made popular in the literature by \citep{chung1989pebbling}. The game is played where pebbles are placed on the vertices in a {\it configuration}, and one vertex is designated the {\it root}. The player may then make {\it pebbling moves} consisting of removing two pebbles from one vertex and placing one pebble onto an adjacent vertex. The goal of the game is to have a pebble reach the root. The {\it pebbling number} of a graph, $G$, denoted $\pi(G)$, is the minimum number of pebbles such that any configuration with that many pebbles can be won by the player no matter the placement of the root. For a general survey of graph pebbling, see \citep{glennSurvey2} or  \citep{glennSurvey}.

Before we continue, let us carefully define our use of Big- and little-O notation. We say $f(n) = O_{n} (g(n))$ if $\lim\sup_{n \to \infty} | f(n) / g(n) |$. Similarly, we say $f(n,d) = O_{n,d} (g(n,d))$ if $\lim\sup | {f(n,d) / g(n,d)} | < \infty$ whenever $n \to \infty$ or $d \to \infty$ or both.

This study concerns the relationship between the pebbling number, $\pi(G)$, the diameter, $d$, and the number of vertices, $n$. Recall that the {\it diameter} of an undirected graph is the maximum length of the shortest path between all pairs of vertices.  It is known that if $G$ is an undirected graph on $n$ vertices with diameter 2, Clarke, Hochberg, and Hurlbert \citep{CHH} showed that either $\pi(G) = n$ or $n + 1$ and classified all graphs with $\pi(G) = n+1$. If $\pi(G) = n$, $G$ is said to be {\it Class-0}, otherwise, it is {\it Class-1}. In the case that $G$ has vertex-connectivity 3 as well, then $\pi(G) = n$ \citep{CHH}. For the case that $G$ has diameter 3, Postle, Streib, and Yerger showed that $\pi(G) \le \left\lfloor\frac{3}{2} n \right\rfloor+2$, and further this bound is sharp; they also showed that for diameter 4 graphs $\pi(G) \le \lceil \frac32 n \rceil + O_{n}(1)$~~\citep{PSY}.  More generally, for an undirected graph with diameter $d$, Bukh showed that the pebbling number is bounded by $\pi(G) \le (2^{\lceil\frac d2 \rceil}-1)n + O_n(\sqrt{n})$ \citep{bukh2006maximum}. Further, there is a graph where
$ \displaystyle \frac{2^{\lceil \frac d2 \rceil}-1}{\lceil \frac{d}{2} \rceil} n + O_{n,d}(1) \le \pi(G)$.  The upper bound has been improved by several subsequent works by Postle, Streib and Yerger where $\pi(G) \le (2^{\lceil\frac d2 \rceil}-1)n  + 16^d + 1$~~\citep{PSY}.

%\franklin{To John: We need to combine the Intro/Prelim or move things around. IMO, the Intro should be the "overview" that a non-pebbling researcher can read and be impressed.}

% \john{This paragraph needs a rewrite as there were several instances where things were defined twice.}

% \franklin{Fixed}

Our study moves graph pebbling into the realm of directed graphs. Indeed, we will assume that all graphs are directed. We use the following notation. Let $D = (V, E)$ be a directed graph where $V$ is the set of vertices and $E$ is the set of directed edges (or arcs). For a directed graph, we will take $E$ to be a subset of ordered pairs $E \subseteq V \times V$ where $(i,j) \in E$ means there is a directed edge from $i \to j$. For our purposes, we will assume that there are no loops (i.e., $(i,i)$ is never an edge). 
%We say $D$ is an {\it oriented graph} if  $i \to j$ then $j \to i$ is not an edge.
In the case where for every unordered pair $\{i,j\}$ in a directed graph is either $i \to j$ or $j \to i$ but not both, we call the graph a \textit{tournament}. 
% In general, we refer any graph graphs in there may be arcs $i \to j$ and $j \to i$ simply as directed graphs.
We focus on loopless directed graphs which may have bidirected edges (i.e., there may be an arc from $i \to j$ {\it and} vice versa); we simply refer to these as directed graphs. At times, we will examine \textit{oriented directed graphs} or \textit{oriented graphs} which do not contain bidirected edges (that is, there is either an arc $i\to j$ or $j\to i$, but not both).  A directed graph has (strong) {\it diameter} $d$ if for every ordered pair of vertices, $(i,j)$ there is a directed path from $i$ to $j$ of length at most $d$. %It should be emphasized that due to the directed nature of our graphs, we must be careful as the shortest path from $i$ to $j$ may be different from that from $j$ to $i$. 
Specifically, when a directed graph has diameter $2$, for any pair of vertices $i, j$, either $i \to j$ or there is a vertex $k$ with $i \to k \to j$. %This brings us to the important fact we will use throughout this article. 
We emphasize that in the case of directed graphs, the shortest path from $i$ to $j$ may contain a different set of vertices than the shortest path from $j$ to $i$.

The {\it strong vertex connectivity} of a directed graph, $D$, is the minimum number of vertices needed to remove from the graph so that for some ordered pair of vertices $(i,j)$ there is no directed path from $i$ to $j$ (however, there could still be a directed path from $j$ to $i$). Hereafter, we will simply call it the {\it strong connectivity} of $D$. If the minimum number of vertex-disjoint directed paths between any pair of vertices in a directed graph $G$ is $k$, we say $G$ is {\it $k$-strongly connected} or has strong connectivity $k$.

Pebbling for directed graphs works analogously for directed games where a pebbling move can be made using arc $i \to j$ by removing 2 pebbles from $i$ and placing one pebble on $j$. The reverse move from $j$ to $i$ can only be made if there is arc $j \to i$. Formally, a configuration is viewed as a function $C(G): V(G)\to \mathbb{N}$. The {\it size} of a configuration is the total number of pebbles, denoted $|C|$. When a root $r$ is specified, a configuration is $r$-{\it solvable} (or just {\it solvable} when $r$ is clear) if there is a sequence of pebbling moves (perhaps zero moves) such that a pebble reaches $r$; otherwise, the configuration is {\it unsolvable}.

Our contribution are analogous to the results in \citep{CHH}, demonstrating a relationship between $\pi(G)$, the diameter, and connectivity, but for directed graphs. We show that some of the results from the undirected case port into the directed case, while others, do not. We prove the following:

%\franklin{This list needs to be changed/corrected/etc. in line with what our results actually are}
\begin{itemize}
\item for an oriented graph $D$ of order $n$ with diameter 2, $\pi(D)= n$ or $ n+1$ (Theorem \ref{noBiN+1}, Section~\ref{sec:boundsStrongDiameter2});
\item if an oriented graph $D$ has strong diameter 2 and $\pi(D) = n+1$, it belongs to a family of directed graphs with a very particular structure (Sections \ref{Class1} and \ref{F});
%\item For an oriented graph $D$ of order $n$ with diameter 2 and strong connectivity $2$, $\pi(D) = n$ (Theorem~\ref{sdsc2}, Section~\ref{sec:boundsStrongDiameter2})
\item for a directed graph $D$  with diameter 2, $\pi(D) < \frac{3}{2}n$, and further, this bound is sharp (Theorem~\ref{mixed2}, Section~\ref{noBiN+1}); and 
\item for a directed graph $D$ with diameter $d, \pi(D) \le 2^d n / d + g(d)$ for some function of only $d$, $g(d)$ and there is an infinite family of a directed graphs with diameter $d$ and $n$ unbounded, where for any $D'$ in the family, $\pi(D') \geq (2^{d-1}-1) \left\lfloor \frac{n-1}{2} \right\rfloor + 2^{{({n-2}) ( {\rm mod~} d})} -1$ (Theorems~\ref{dboundsharp} and \ref{dboundupper}, Section \ref{upper}) %\franklin{This needs to be more exact}.
%\item For a tournament $T$ of order $n$ with diameter 2, $\pi(T) = n$, (Theorem~\ref{tournament2}, Section~\ref{noBiN+1}).
\end{itemize}
%Of particular note is Theorem~\ref{sdsc2}. \john{Explain why this is important.} \franklin{}

%\section{Structure Lemmas}\label{sec:prelim}
%
%\franklin{Need a preface here}
%
%\begin{prop}\label{3orless}
%Let $D$ be a directed graph with diameter 2. Then in any configuration of pebbles that is unsolvable, no vertex has more than 3 pebbles.
%\end{prop} 
%
%\begin{proof}
%Suppose a vertex had 4 pebbles. Call that vertex $v$, and let $r$ be the root. Since $D$ has strong diameter $2$, either $v \to r$ or there is vertex $k$ with $v \to k \to r$. In either case, the pebbles on $v$ can be used to get at least one pebble on $r$.
%\end{proof}

%
%\john{Define solvable and unsolvable.} 
%We say that we can pebble a vertex $x$ in a graph by applying pebbling steps until at least one pebble is on vertex $x$. Let $\pi(G)$ be the minimum number of pebbles such that given any configuation of pebbles on the vertices, it is possible to apply pebbling steps to pebble any vertex in $V(G)$. This is called the \textit{pebbling number}. 

\section{Bounds for Directed Graph with Strong Diameter 2} \label{sec:boundsStrongDiameter2}

% \john{Say something about this section.}

In this section, we focus on creating bounds for directed and oriented graphs with strong diameter $2$.

\begin{prop}\label{3orless}
Let $D$ be a directed graph with strong diameter 2. Then in any configuration of pebbles that is unsolvable, no vertex has more than 3 pebbles.
\end{prop} 

\begin{proof}
Suppose a vertex had 4 pebbles. Call that vertex $v$, and let $r$ be the root. Since $D$ has strong diameter $2$, either $v \to r$ or there is vertex $k$ with $v \to k \to r$. In either case, the pebbles on $v$ can be used to get at least one pebble on $r$.
\end{proof}

\begin{theorem}\label{noBiN+1}
Let $D$ be an oriented graph with strong diameter $2$ on $n$ vertices. Then $\pi(D)\leq n+1$. 
\end{theorem}

%Two vertices in $A$ with $2$ or $3$ pebbles each cannot have an arc to the same vertex in $B$, otherwise both of those vertices can move a pebble to that vertex in $B$, and those pebbles can be moved to $r$.

% % \franklin{This proof needs to be tightened up a bit. But I think it is all there.}
% \franklin{Compactified it!}

\begin{proof}
We argue by contradiction, supposing that $C$ is a configuration with $n+1$ pebbles and there is no sequence of pebbling moves that moves a pebble to a root $r$. By Proposition~\ref{3orless}, no vertex in $D$ can have more than $3$ pebbles. Partition $V(D)$ into $\{A,B,\{r\}\}$ where all vertices in $A$ are distance $2$ from vertex $r$ and all vertices in $B$ are distance $1$ from $r$. If $r$ has a pebble, we are done, so for the remainder of this proof, we assume $C(r) = 0$. By employing the same argument as above, all vertices in $B$ have at most $1$ pebble. 
By the pigeonhole principle, if $A$ is empty, then one vertex in $B$ has at least two pebbles, but then we can pebble $r$. So there is at least one pebble on a vertex in $A$. 

Partition $A$ further into $\{A_3,A_2,A_1,A_0\}$ such that each vertex in $A_3$, $A_2$, $A_1$, and $A_0$ contain $3$, $2$, $1$, and $0$ pebbles respectively. Notice that each vertex in $A$ is joined to at least one vertex in $B$. We further partition $B$ into $\{B_3,B_2,B_1,B_0\}$ as follows. For each $a \in A$, choose a path $a \to b \to r$ for some $b \in B$. For $a \in A_2 \cup A_3$, these paths must necessarily use different $b$'s, as otherwise two pebbles could reach $b$ which can be used to pebble $r$. We define $B_3$ to be the set of $b$ chosen for each $a \in A_3$ and $B_2$ to be those chosen for each $a \in A_2$. Hence, $|A_3| = |B_3|$ and $|A_2| = |B_2|$. Let $B_1$ and $B_0$ be the vertices of $B$ with exactly 1 or 0 pebbles respectively.  Notice that any vertex that lies on a directed path of length $2$ from a vertex in $A_2\cup A_3$ to $r$ must have $0$ pebbles. Therefore,  $\{B_3,B_2,B_1,B_0\}$ is indeed a partition of $B$.

Notice that each vertex in $B_0$ has exactly $0$ pebbles and it is possible there are vertices in $B_0$ that have arcs from any vertex in $A$.
To ensure there are no vertices with $4$ pebbles after one pebbling move, there are no arcs between two vertices in $A_3$, no pair of vertices in $A_2\cup A_3$ has an arc to the same vertex in $A_2$, and there are no arcs from a vertex in $A_2$ to a vertex in $A_3$. 
% Hence, there are at most $|A_2|$ directed arcs between pairs of vertices in $A_2\cup A_3$. \franklin{Is this previous sentence used at all?}
Since $B_3\cup B_2\cup B_0$ have exactly $0$ total pebbles, the number of pebbles on $A_3$, $A_2$, $A_1$, and $B_1$ sum up to $n+1$. Thus,
\begin{eqnarray}
\notag
3|A_3|+2|A_2|+|A_1|+|B_1| &=& n+1 = 2 + |A_3|+|A_2|+|A_1|+|A_0|+|B_3|+|B_2|+|B_1|+|B_0| \\
\notag
2|A_3|+|A_2| &=& 2 + |A_0|+|B_3|+|B_2|+|B_0|  \\
\notag
2|A_3|+|A_2| &=& 2 + |A_0|+|A_3|+|A_2|+|B_0|  \\
\label{b0small}
|A_3| &=& 2 + |A_0|+|B_0|  \\
\label{a3size}
|A_3| &\ge& 2. 
\end{eqnarray}
The final line follows as $A_0$ and $B_0$, in the worst case, could be empty.

That is, there are at least two vertices with three pebbles. We will show this causes a contradiction.

Let $x_1,x_2\in A_3$, and let $y_1,y_2\in B_3$ be the corresponding vertices of $x_1$ and $x_2$ in $B_3$.
We aim to have $|A_0\cup B_0|$ be as small as possible. There are directed paths of length at most two, $P_1$ and $P_2$, from $x_2$ to $y_1$ and from $x_1$ to $y_2$ respectively. Such paths must exist since $D$ has strong diameter $2$. There is no edge from $x_2$ to $y_1$ or $x_1$ to $y_2$, as otherwise we can pebble $r$. Thus, we can take $P_1$ to be $x_1 \to q_1 \to y_2$ and $P_2$ to be $x_2 \to q_2 \to y_1$. Note that it must be the case that $q_1, q_2 \in A_0\cup B_0\cup\{y_1,y_2\}$,
% ~~\franklin{I deleted a $\cup\{y_1,y_2\}$ here, is it necessary?}. 
$q_1$ and $q_2$ are distinct (as we could pebble $r$ otherwise), and $\{q_1,q_2\}\neq \{y_1,y_2\}$ (as we have bidirected edges otherwise). Note that we can make the same argument for each pair in $A_3$.  Therefore, $2 {|A_3| \choose 2} \leq  |A_0\cup B_0| =  |A_0| + |B_0| $. However, by line (\ref{a3size}), $|A_3| = 2 + |A_0|+|B_0|$. Since ${{x+2} \choose 2} > 2 + x$ for any real number $x$, we reach a contradiction by taking $x = |A_0|+|B_0|$.
\end{proof}

When bidirected edges are allowed in $D$, the bound in Theorem~\ref{noBiN+1} changes quite drastically as we see in the next theorem. 

\begin{theorem} \label{mixed2}
Let $D$ be a diameter $2$ directed graph of order $n$. Then $\pi(D) <\frac{3}{2}n $. Further, this bound is sharp.
\end{theorem}

 %\franklin{Proof is done.}

\begin{figure}[!b]
\begin{center}
\includegraphics[width=0.3\textwidth]{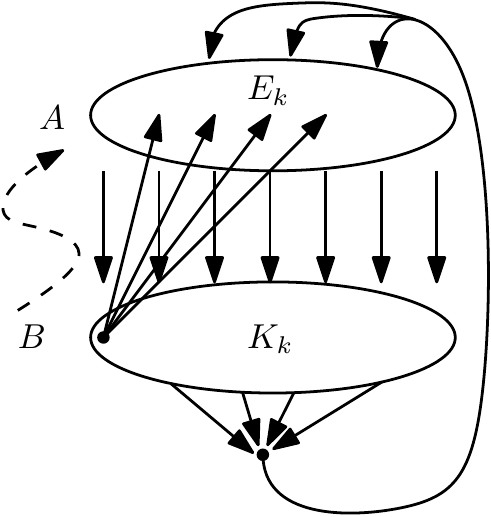}
\end{center}

\caption{A construction for extremal examples in Theorem \ref{mixed2}. The top is an empty graph on $k$ vertices, and the bottom is a complete graph on $k$ vertices. The %ski-sloping 
dotted line indicates that possible arcs from $B$ to $A$ are present except for the matching of arcs indicated from $A$ to $B$.}
\label{mixed2const}
\end{figure}

\begin{proof}
Choose a root $r$ and let $C$ be an unsolvable configuration of  pebbles. Then by Proposition~\ref{3orless}, all vertices in $D$ have at most 3 pebbles. Let $A_{2+}$ be the set of vertices with at least $2$ pebbles. If $r$ is in the out-neighborhood of some $a \in A_{2+}$, then $r$ can be pebbled directly using the two pebbles on $a$. Hence, we can assume that the out-neighborhood of each vertex of $A_{2+}$ does not contain $r$. For each $a \in A_{2+}$, choose a $b_a$ such that there is a path $a\to b_a \to r$. 
% \john{Use the same notation here in the previous theorem.} 
If $b_u = b_w$ for two distinct $u, w \in A_{2+}$, then $C$ is $r$-solvable by moving two pebbles onto $b_u (=b_w)$ from $u$ and $w$, and so $r$ can be pebbled from $b_u$. Further, if $b_u$ has a pebble, then $C$ is $r$-solvable along the path $u\to b_u \to r$. Therefore, each pair $a, b_a$ has at most $3$ pebbles, and further, any single vertex in $D$ that is not in such a pair has at most $1$ pebble (otherwise, it would be in $A_{2+}$). To optimize the number of pebbles in an unsolvable configuration, one should have as many pairs as possible, as each pair allows for 3 pebbles whereas non-paired vertices allow the placement of at most 2 pebbles. Since there are at most $\lfloor (n-1)/2 \rfloor$ pairs, it must have at most $3 \lfloor (n-1)/2 \rfloor +1$ pebble where the extra pebble may result from a single vertex not in a pair. Since $3 \lfloor (n-1)/2 \rfloor +1 < \frac{3}{2}n$, this completes the proof. %\john{Why does $3 \lfloor (n-1)/2 \rfloor < \frac{3}{2}n$ imply $C$ is solvable?}

To see this is sharp, choose a positive integer $k$, and consider the undirected complete bipartite graph $K_{k,k}$ with parts $A$ and $B$. Choose a perfect matching, and orient those edges from $A$ to $B$, and orient all other edges in reverse. Add a complete bidirected graph to the vertices of $B$ and finally add a vertex $r$ where all vertices in $B$ have an arc to $r$ and all vertices in $A$ have an arc from $r$. This construction is illustrated in Figure \ref{mixed2const}.

We claim this graph has strong diameter 2. To find a path between parts, use $A \to B \to r \to A$ where a path from $A$ to $B$ must use the edges of the complete graph within $B$. To find a path between vertices in $A$, use the matching edge and take a non-matching edge to the desired vertex, and within $B$, simply use the edges in the complete bidirected graph within $B$. Finally, consider the configuration where each vertex in $A$ has 3 pebbles. Then, each vertex has exactly one pebbling move available, and further, the vertices that can receive a pebble each receive a pebble from a unique source. Therefore, each vertex in $B$ will have at most 1 pebble, preventing a pebble from ever reaching $r$. It follows that the pebbling number of $G$ is at least $3k$.
% \john{What is $k$?} \franklin{$k$ is well defined at the start of the previous paragraph!}. 
However, the result guarantees that the pebbling number is at most $\frac{3}{2} (2k+1) = 3k + \frac{3}{2}$. Therefore, the pebbling number is necessarily $3k+1$, the maximum possible.
\end{proof}

\section{Classification of Class-1, 2-connected, diameter 2, oriented graphs} \label{Class1}

In this section, we develop a classification for 2-connected, diameter 2 oriented graphs that are Class-1. Our classification is similar to the case for undirected graphs (see \citep{CHH}); however, the restriction to \emph{oriented} directed graphs makes for important, yet subtle distinctions. As a consequence, all other $2$-connected, diameter $2$ oriented graphs must be Class-0. 

%In the previous section, we showed that with regards to pebbling, directed graphs are both very similar and very different to undirected graphs in the case of small diameter. Previous work by  \john{this sentence is incomplete.}

% 
%
%
%\john{Can the following theorem be used in Section 2?}
%\franklin{Is this question still valid?}

\begin{theorem}\label{two3no2}
Let $D$ be a $k$-connected strong diameter $2$ oriented graph for $k\geq 2$. If $D$ is Class $1$, then for any unsolvable configuration $C$ of size $|V(D)|$ we have $|A_2|=0$ and $|A_3|=2$ or $D$ is a directed $3$-cycle.
Furthermore, there are exactly four vertices with $0$ pebbles. 
\end{theorem} 

% \franklin{We should mention this theorem/proof is similar to \citep{CHH}, but emphasize why it is different. We should be careful as to not plagiarize their work.}

\begin{proof}
Let $C$ be an unsolvable pebbling configuration with $n$ pebbles placed on the vertices of a directed graph $D$ such that we let $A_2=\{v\in V(D)\,:\, C(v)=2\}$, $A_3=\{v\in V(D)\,:\, C(v)=3\}$, and $Z=\{v\in V(D)\,:\, C(v)=0\}$. For two sets $X,Y\subseteq V(G)$, let $V(X,Y)$ denote the set of vertices with an arc from a vertex $x\in X$ and an arc to a vertex $y\in Y$, $x\neq y$. If $Y=\{v\}$, then we simply use $V(X,v)$ instead of $V(X,\{v\})$.

Let $r$ be the root of $D$, and let $|A_2|=s$ and $|A_3|=t$. We will first find a lower bound on the size of $Z$. This lower bound on $Z$ will, in effect, bound $A_2$ and $A_3$. 
By Proposition~\ref{3orless}, no vertex in $C$ can have $4$ pebbles nor can any vertex have $4$ pebbles after making any number of pebbling moves from $C$.
We consider the sets of vertices in $V(A_2,A_3),V(A_3,A_3),V(A_2,r),V(A_3,r)$. 

If there is a vertex in $V(A_2,A_3)$ with $1$ or more pebbles on it then we can make pebbling moves to get $4$ pebbles on some vertex, a contradiction by Proposition~\ref{3orless}. Similarly, all vertices in $V(A_2\cup A_3,A_3)$ have $0$ pebbles. 
Since there cannot be two vertices in $A_2$ with the same out-neighbor in $V(A_2,A_3)$ and to ensure that the distance from a vertex in $A_2$ to a vertex in $A_3$ is at most $2$, $|V(A_2,A_3)|\geq st$. %\john{Franklin, do you agree that this should be $\geq$ rather than $=$?} 
If $V(A_2,A_3)\cap V(A_3,A_3)\neq \varnothing$ then there would be a way to get $4$ pebbles on some vertex. 
If $V(A_2,r)\cap V(A_3,r)\neq \varnothing$, $V(A_3,A_3)\cap V(A_3,r)\neq \varnothing$, $V(A_3,A_3)\cap V(A_2,r)\neq \varnothing$ , $V(A_2,A_3)\cap V(A_3,r)\neq \varnothing$, $V(A_2,A_3)\cap V(A_2,r)\neq \varnothing$ then it would be possible to move two pebbles onto a vertex with an out-arc to $r$, thus we can pebble $r$, a contradiction. Thus $V(A_2,A_3),V(A_3,A_3),V(A_2,r),V(A_3,r)$ are all mutually disjoint. To ensure $r$ cannot receive a pebble from $A_2$ or $A_3$, no vertex in $V(A_2,r)\cup V(A_3,r)$ can have one or more pebbles and $|V(A_2,r)|\geq s$, $|V(A_3,r)|\geq t$. For the same reasons, $V(A_3,A_3)$ cannot contain any vertex with one or more pebbles and $|V(A_3,A_3)|\geq \binom{t}{2}$. 
Since $G$ is Class-1, then the number of pebbles in $G$ is $n$ where $A_1$ is the set of vertices with exactly one pebble. Thus the number of vertices in $G$ can also be found as $|Z|+|A_1|+|A_2|+|A_3|$, and so $|A_1|+2|A_2|+3|A_3|=|Z|+|A_1|+|A_2|+|A_3|$ implies that $|Z|+|A_2|+2|A_3|$. 
Since all of the vertices in $V(A_2,A_3)\cup V(A_3,A_3)\cup V(A_2,r)\cup V(A_3,r)\subseteq Z$ and all of these sets are disjoint,
\begin{align}\label{zinequality}
st+\binom{t}{2}+s+t+1\leq |Z|=s+2t.
\end{align}
%\john{Look up why this is true. I may need help on this one, Franklin.} \franklin{I resolved this, correct?}
From \eqref{zinequality}, we can derive the inequality
\begin{align}\label{zinequality2}
t^2-(3-2s)t+2\leq 0.
\end{align}
Hence $3\geq2s$, and so $s\leq 1$. 

If $s=1$ then $t^2-t+2=\left(t-\frac{1}{2}\right)^2+\frac{7}{4}>0$ which contradicts \eqref{zinequality2}. So we suppose that $s=0$. Then $t^2-3t+2\leq 0$ implies $t=1$ or $2$. Further suppose that $t=1$. Then $|Z|=2$ and without loss of generality we suppose $v_t$ is the one vertex in $A_3$ and $Z=\{v,r\}$. Necessarily, $v_t\to v$ and $v\to r,$ are arcs in $D$. 
%If $D$ is the graph with vertex set $T\cup Z$ and edges $v_t\to v,v\to r,r\to v_t$, then $D$ is a directed $3$-cycle. 
Since $D$ is $2$-connected, there are two disjoint directed paths from $v_t$ to $r$. The other path will contain only vertices from $A_1$ and both $v_t,r$. But then we can pebble $r$ with the path through $A_1$, a contradiction. Thus, our result follows.
\end{proof}

%\franklin{This proof seems longer than it needs to be}

We will next define the family of $2$-connected strong diameter $2$ Class $1$ directed graphs, $\F$. We will show that these are the only $2$-connected strong diameter $2$ graphs that are Class $1$. 
See Figure~\ref{class1graphs} for a pictorial representation of the family $\F$. 
% \john{Explain the edges in the this figure. Would it be better to put the beginning of Section 4 here?}

\begin{figure}[htb]
\begin{center}
\includegraphics{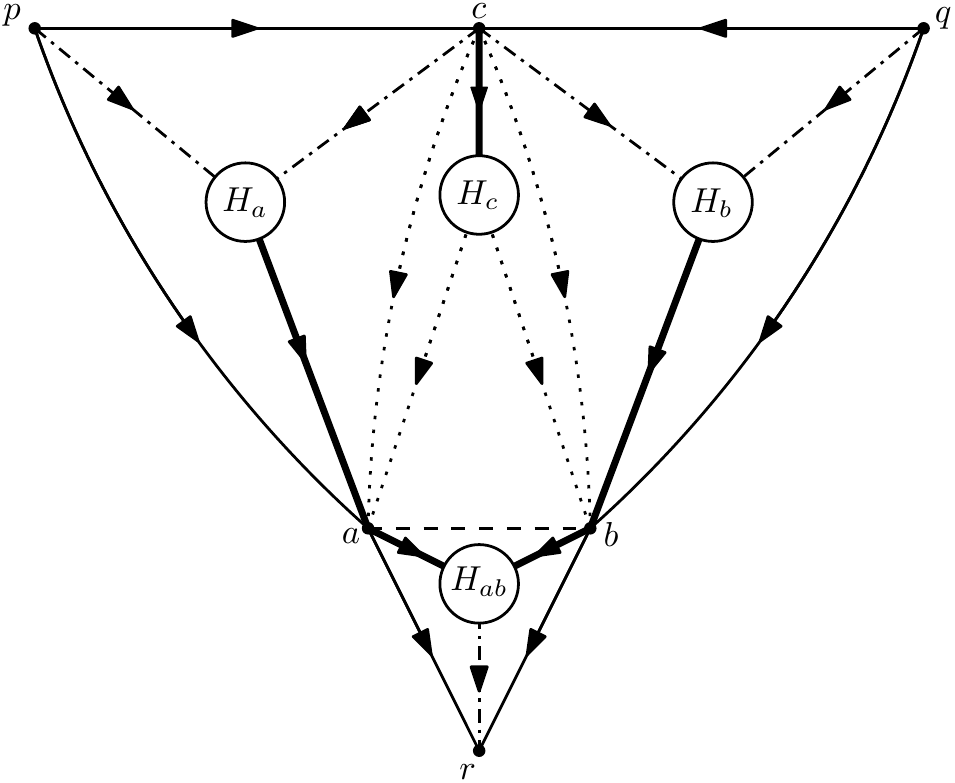}
\end{center}
\caption{Schematic diagrams for the necessary arcs for oriented graphs in $\F$. Arbitrary upward arcs are allowed. }\label{class1graphs}
\end{figure}

\begin{defi}
Let $\mathcal{F}$ be the set of 2-connected oriented graphs with strong diameter 2 such that 
\begin{itemize}
\item there is a 6-cycle with orientation, $p \to c \leftarrow q \to b \to r \leftarrow a \leftarrow p$;
\item any directed path from $p$ to $r$ either contains $a$ or both $c$ and  $b$;
\item any directed path from $q$ to $r$ either contains $b$ or both $c$ and  $a$; and
\item any directed path from $c$ to $r$ either contains $a$ or $b$.
\end{itemize}
\end{defi}

In Figure~\ref{class1graphs}, solid lines are arcs and bold lines mean all the possible arcs between two sets of vertices are present. The dashed line arcs indicate that there is an arc $a\to b$, $b\to a$, or possible none depending on the adjacencies between $H_c\cup \{c\}$ and $\{a,b\}$. 
The dotted arcs from $c$ to $a$ and $b$ indicate that there is either an arc $c\to a$ or $c\to b$. Similarly for each vertex in $H_c$ there is either an arc to $a$ or $b$. 
The dashed-dotted arcs indicate there are some arcs between the indicated vertices but possibly not all arcs. 
There must be enough edges to satisfy the $2$-connected property of these graphs in $\mathcal{F}$. We give justification for these arcs in Section~\ref{F}. It is possible to have arcs going in the ``upward'' direction in Figure~\ref{class1graphs}.

Furthermore, we define the following for graphs in $\mathcal{F}$:

\begin{itemize}
\item[] $H_a$ is the set of all intermediary vertices $v_a$ such that there is a directed path from $p$ to $a$ without using vertex $c$.
\item[] $H_b$ is the set of all intermediary vertices $v_b$ such that there is a directed path from $q$ to $b$ without using vertex $c$.
\item[] $H_c$ is the set of all intermediary vertices $v_c \not\in H_a \cup H_b$ such that there is a directed path from $c$ to $a$ or from $c$ to $b$.
\item[] $H_{ab}$ is the set of all intermediary vertices $v_{ab}$ such that there is a directed path from $a$ or $b$ to $r$ that does not contain the other.
\end{itemize}

\begin{theorem}
If $D \in \mathcal{F}$, then $D$ is Class-1.
\end{theorem}

\begin{proof}
By Theorem~\ref{two3no2}, there are two vertices with $3$ pebbles, say $p$ and $q$. Let $a$, $b$, $c$, and $r$ have $0$ pebbles. All other vertices have exactly one pebble. It is clear we cannot make pebbling moves from $p$ to $r$ without using the pebbles on $q$ and vice versa. 
Thus we need to get two pebbles on either $a$ or $b$. Any directed path from $p$ to $b$ must contain $a$ or $c$, so there is no set of pebbling moves that will move the pebbles from $p$ to $b$. Similarly there is no way to pebble $a$ using pebbling moves starting at $q$. The only vertices that can be pebbled using pebbling moves starting at $p$ and $q$ is $c$. 
Since any path from $c$ to $r$ must go through $a$ or $b$, we can get at most one pebble on $a$ or $b$. Thus we cannot pebble $r$.
\end{proof}

\subsection{Properties of $\mathcal{F}$} \label{F}

% \john{Explain the purpose of this section.}

Previously, we defined the class of oriented graphs $\mathcal{F}$. However, this definition leads to a highly structured graph beyond the definition. In the results below, we derive many further properties of $\mathcal{F}$.

\begin{prop}
If $D \in \mathcal{F}$, then for every $v_a \in H_a$, there is an arc $v_a \to a$, and likewise, for every $v_b \in H_b$, there is an arc $v_b \to b$.
\end{prop}
\begin{proof}
By symmetry, we will prove that for every $v_a \in H_a$, there is an arc $v_a \to a$. Suppose there is a vertex $v_a \in H_a$ without an arc $v_a \to a$. Then, we claim there is no directed path from $v_a$ to $r$ of length 2 or less. Indeed if $v_a \to w \to r$ for some vertex $w \ne a$, then there is a path $a \to \cdots \to v_a \to w \to r$ that does not contain $a$ nor $c$, violating the definition of $\mathcal{F}$.
\end{proof}

\begin{prop}
If $D \in \mathcal{F}$, then for every $v_a \in H_a$, there is no arc $v_a \to b$, and likewise, for every $v_b \in H_b$, there is no arc $v_b \to a$.
\end{prop}
\begin{proof}
By symmetry, we will prove for every $v_a \in H_a$, there is no arc $v_a \to b$. Suppose $v_a \to b$, then there a path $p \to \cdots \to v_a \to b \to r$ that does not contain $a$ nor both $b$ and $c$, a contradiction.
\end{proof}

\begin{prop} \label{8}
If $D \in \mathcal{F}$, then for every $v_{ab} \in H_{ab}$, there is an arc $a \to v_{ab}$ and $b \to v_{ab}$.
\end{prop}
\begin{proof}
By symmetry, we will prove that for every $v_{ab} \in H_{ab}$, there is an arc $a \to v_{ab}$. Suppose there is a vertex $v_{ab} \in H_{ab}$ without an arc $a \to v_{ab}$. Choose a vertex $h \in H_{a}$. Then, we claim there is no directed path from $h$ to $v_{ab}$ of length 2 or less. Indeed, if $h \to w \to v_{ab}$ for some vertex $w \ne a$, then there is a path $p \to \cdots \to h \to w \to v_a \to \cdots \to r$ that does not contain $a$ nor $c$, violating the definition of $\mathcal{F}$.
\end{proof}

\begin{prop}\label{ctoaorctob}
For every $v_c \in H_c$, there is an arc $v_c \to a$ or $v_c \to b$.
\end{prop}
\begin{proof}
Suppose there is no arc from $v_c \to a$ nor $v_c \to b$, then there must be a vertex $w \ne a, b$ such that $v_c \to w \to r$. In which case, there is a path $c \to \cdots \to v_c \to w \to r$ that avoids $a$ and $b$, which violates the definition of $\mathcal{F}$.
\end{proof}

\begin{prop}
If $D \in \mathcal{F}$, then for every $v_{a} \in H_{a}$ and every $v_c \in H_c$, there is no arc $v_a \to v_{c}$. Likewise, for every $v_{b} \in H_{a}$ and every $v_c \in H_c$, there is no arc $v_b \to v_{c}$.
\end{prop}
\begin{proof}
By symmetry, we will prove that for every $v_{a} \in H_{a}$ and every $v_c \in H_c$, there is no arc $v_a \to v_{c}$. Suppose there is an arc $v_a \to v_c$. Then, by Proposition \ref{ctoaorctob}, there is an arc from $v_c \to a$ or from $v_c \to b$. If $v_c \to a$, then it follows that $v_c \in H_a$ as there is a path $a \to h_a \to h_c \to r$, and hence, it must be the case that $v_c \to b$. Therefore, there is a path $p \to \cdots \to v_a \to v_c \to b \to r$ which goes though $b$ that does not go through $a$ nor $c$, violating the definition of $\mathcal{F}$.
\end{proof}

\begin{prop}
If $D \in \mathcal{F}$, then either $c \to a$ or $c \to b$.
\end{prop}
\begin{proof}
Suppose not, then either $c \to r$ or $c \to w \to r$ for some vertex $w \ne a,b$. However, these both violate the definition of $\mathcal{F}$.
\end{proof}

\begin{prop}
If $D \in \mathcal{F}$, then for any $v_c \in H_c$, either $v_c \to a$ or $v_c \to b$.
\end{prop}
\begin{proof}
Suppose not, then either $v_c \to r$ or $v_c \to w \to r$ for some vertex $w \ne a,b$. This means there is a path $c \to \cdots v_c \to r$ (for the first case) or a path $c \to \cdots \to v_c \to r$ (in the second case) neither of which go through either $a$ or $b$. Hence, these cases both violate the definition of $\mathcal{F}$.
\end{proof}

\begin{prop} 
If $D \in \mathcal{F}$, if $c \not \to a$, then $b \to a$, and likewise, if $c \not \to b$, then $a \to b$.
\end{prop}
\begin{proof} 
By symmetry, we will prove that if $c \not \to a$, then $b \to a$. Suppose $c \not \to a$ and also $b \not \to a$. Since there must be a path of length 2 from $q$ to $a$, there must be a vertex $w \ne b, c$ such that $q \to w \to a$. In which case, there is a path $q \to w \to a \to r$ which goes though neither $b$ nor both $c$ and $a$, thus violating the definition of $\mathcal{F}$.
\end{proof}

\begin{cor}
Either $c \to a$ or $c \to b$.
\end{cor}
\begin{proof} 
If both $c \not \to a$ and  $c \not \to b$, then by the previous result, we have $a \to b$ and $b \to a$, violating the fact that $D$ is an oriented graph. 
\end{proof}

\begin{prop}
If $D \in \mathcal{F}$, then the induced subgraph on the vertices $H_{ab} \cup \{r\}$ has strong diameter 2.
\end{prop}

%\john{The following proof proves $D$ is unsolvable, but we were supposed to show $D$ has strong diameter 2 $\ldots$right? Franklin needs to resolve this.}
\begin{proof} 
Suppose not. Then, since $D \in \mathcal{F}$, $D$ has strong diameter 2. Hence, for some pair of vertices $x,y \in H_{ab} \cup \{r\}$, there is a directed path $x \to v \to y$ for some $v \not\in H_{ab} \cup \{r\}$. Further, by Proposition \ref{8} (or by definition of $\mathcal{F}$ if $y=r$), there is an arc $a \to x$ and $b \to x$, therefore $x \not\to a$ and $x \not\to b$, and so $v \ne a$ nor b. Hence $v \in \{p, q, c\} \cup {H_a} \cup H_b \cup H_c$. If $v = p, q$ or $c$, then there is a path to $r$ without $a$ or $b$. If $v \in H_a$, then there is a path from $p \to w \to \cdots \to v \to y \to r$ where $w \ne c$ that avoids both $a$ and $c$, and similarly if $v \in H_b$, there is a directed path from $q$ to $r$ that avoids both $b$ and $c$. Finally, if $v \in H_c$, then there is a path $c \to\cdots \to v \to y \to\cdots \to r$ that avoids both $a$ and $b$. All of these paths violate the definition of $\mathcal{F}$, a contradiction.
\end{proof}

%Any path from $p$ to $r$, $q$ to $r$, or $c$ to $r$ must go through $a$ or $c$, $b$ or $c$, or $a$ or $b$. Any path from $p$ to $b$ must contain $c$ or $b$. thus, we cannot get a pebble on $b$ from $p$. Similarly, we cannot get a pebble on $a$ from $q$.
%
%Let $p$ and $q$ have $3$ pebbles, $a,b,c,$ and $r$ have $0$ pebbles, and all other vertices have one pebble.
%It is clear we cannot make a pebbling moves form $p$ to $r$ without using the pebbles on $q$ and vice versa.
%Any directed path from $p$ to $b$ must contain $a$ or $c$, so there is no way to pebble $b$ with pebbling moves only from $p$. 
%Similarly, there is no way to pebble $a$ with the pebbles on $q$ alone. The only vertices that $p$ and $q$ can both pebble individually is $c$. Since any path from $c$ to $r$ must go through $a$ or $b$, we can get at most one pebble on $a$ or $b$. Thus, we cannot pebble $r$.
%\end{proof}

%\section{Ideas to Move Forward}

%\begin{itemize}
%	\item Go through Glenn's survey paper. Which if any of those results will produce interesting results only for the directed case? Especially for diameter $2$. Orientations of hypercube? Can we do something with it? Orientations of certain graphs? Spread of graphs?
%	\item Optimal pebbling in directed graphs? How is it better than optimal pebbling in undirected graphs?
%\end{itemize}

%\section{Higher diameter}

\section{Bounds On $\pi(G)$ For directed graphs with diameter $d$} \label{upper}

As we did with directed graphs of diameter $2$, we provide bounds for the pebbling number of graphs with diameter $d$ in this section. 
Let $f(n,d)$ be the largest $\pi(G)$ for all strongly connected directed graphs $G$ with $n$ vertices and diameter $d$. It is clear that $f(n,d)\geq \pi(G)$ if $G$ has order $n$ vertices and diameter $d$.

\begin{theorem} \label{dboundsharp}
If $D$ is a directed graph of diameter $d$ on $n$ verteices, then \[ \displaystyle f(n,d) > (2^{d-1}-1) \left\lfloor \frac{n-1}{2} \right\rfloor + 2^{{({n-2}) ( {\rm mod~} d})} -1. \] 
%Furthermore, there is a graph that is sharp.
\end{theorem}

\begin{proof} %\franklin{Adjust proof for generic case}
It suffices to construct a directed graph, $D$, with $n$ vertices and diameter $d$ with $\pi(D) \geq (2^{d-1}-1) \left\lfloor \frac{n-1}{d} \right\rfloor + 2^{{({n-2}) ( {\rm mod~} d})} -1.$

We construct $D$ similar to Figure \ref{mixed2const2}. Let $m = {({n-1}) ( {\rm mod~} d})$. Consider $(d-m)$ copies of the empty graph on $k-1$ vertices $E_{k-1}^1, \ldots, E_{k-1}^{d-m-1}$ and $m$ copies of the empty graph on $k$ vertices $E^{d-m}_{k} \ldots E^{d-1}_{k}$.
%\john{I thought this was supposed to be generalized. Why is this size $k$ or $k-1$?} 
For simplicity going forward, we omit the subscript and call the set of vertices in $E^\ell$ as ``layer $\ell$.'' Add a directed matching between each pair of successive layers: $E^i$ and $E^{i+1}$ for $i=1, \ldots, d-2$. Add a copy of the complete directed graph on $k$ vertices $K_k$ and add a directed matching from $E^{d-1}$ to $K_k$, and for each $i$, add all remaining arcs from $K_k$ to $E^i$ where there is not already a reverse arc to $K_k$. Finally, add a root $r$ where each vertex in the $K_k$ has a directed edge to $r$ and $r$ has a directed edge to all vertices in $E^1$.  We refer to the vertices in $K_k$ as layer $d$ and the root $r$ as layer $d+1$. We now show that $D$ has diameter $d$. Choose any two vertices $x$ and $y$. If $x$ or $y$ is $r$, then there is path of length at most $d$, by either following the matching edges from $x$ to $r$, or by following the matching edges backwards from $r$ to $y$. Hence, we can assume that $x$ nor $y$ is $r$. Suppose $x$ is in layer $i$ and $y$ is in layer $j$. Then we have a path $x \to v_{i+1} \to \ldots \to v_{d-1} \to v_{d} \to y$ where $v_\ell$ is in layer $\ell$, unless $v_{d-1} = y$, in which case, we have a shorter path $x \to v_{i+1} \to \ldots \to v_{d-1} = y$. In both cases each path is length at most $d$. 

We claim that $\pi(D) \geq (2^{d-1}-1) (k-1) + 2^{{({n-2}) ( {\rm mod~} d})} -1 $ . Consider a configuration of  $(2^{d}-1) k$ pebbles by placing $2^{d}-1$ pebbles on each vertex in layer 1; and an additional $2^{{({n-2}) ( {\rm mod~} d})} -1$ pebbles on the vertex $v'$ in layer $d-m+1$ with no matching arc from layer $d-m$. Observe that the only out-arcs from each vertex in layer 1 is a matching into layer 2. Likewise, the only out-arcs from each vertex in layer $i$ is a matching into layer $i+1$ for $i = 1, \ldots, d-1$. Therefore, the only way to get a pebble onto a vertex in layer $d$ is use only successive pebbling moves along a single path $v_1 \to \ldots \to v_d$ or $v' \to \ldots \to v_d$ where $v_\ell$ is some vertex in layer $\ell$. However, since $v_1$ only starts with $2^d-1$ pebbles and $v'$ starts with $2^{{({n-2}) ( {\rm mod~} d})} -1$ pebbles, at most one pebble can ever reach any vertex in layer $d$, and hence, there is no way to pebble $r$.
\end{proof}

\begin{figure}[!h]
\begin{center}
\includegraphics[width=0.3\textwidth]{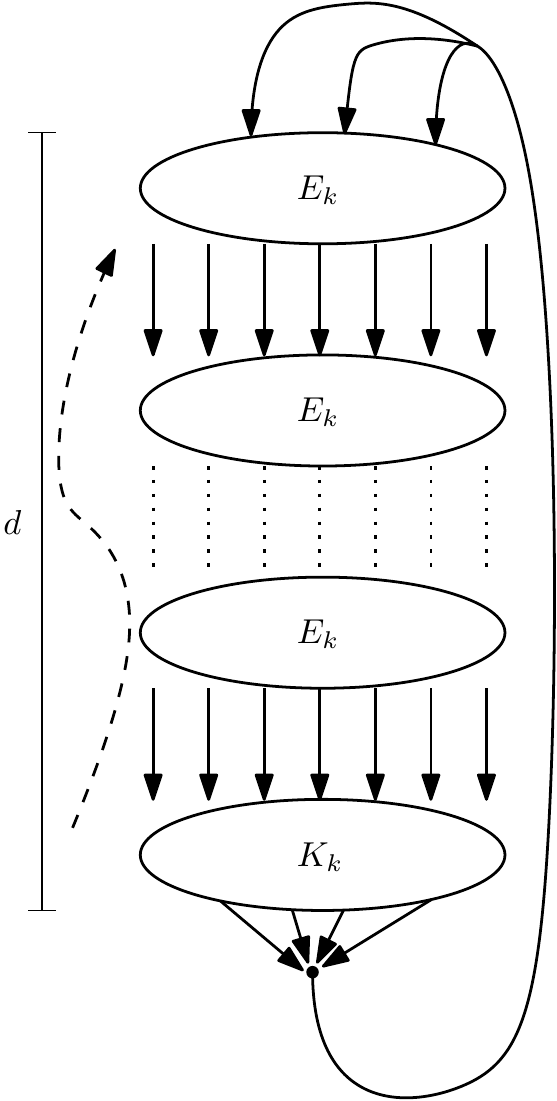}
\end{center}
\caption{A schematic for the construction in Theorem \ref{dboundsharp}. Here, there are $d$ levels each containing $k$ vertices plus a root and each level is joined by a directed matching toward the root. All but the last level are empty graphs and the last level is a complete graph. Additionally, the upwards curve denotes that each vertex of the last level has an arc to every vertex above it, except for the corresponding single vertex in the level immediately preceding it.}
\label{mixed2const2}
\end{figure}

%\john{Define $f(n,d)$.}
% Let $f(n,d)$ be the largest pebbling number among all graphs with $n$ vertices and diameter $d$. We adapt the proof technique in \citep{PSY} for directed graphs.

The next theorem stands in direct contrast to Theorem 7 in \citep{PSY}. 

\begin{theorem} \label{dboundupper}
Let $G$ be a strongly connected directed graph of order $n$ with diameter $d$. For some fixed positive integer $d$, 
\[
f(n,d)\leq n\left(\frac{2^d}{d}-1\right)+2^{4d+1}\left(1-\frac{1}{d}\right).
\]
\end{theorem}

\begin{proof}
Let $(G,r)$ be a rooted graph with configuration $P$ and suppose that there is no sequence of pebbling moves that moves a pebble to a root $r$. 
For each vertex $v$, let $X(v)=P(v)-\left\lfloor\frac{2^d}{d}\right\rfloor+1$.
The following proof is broken into two parts, depending on the size of $H$ where $v\in H$ if $P(v)\geq \frac{2^d}{d}$. We call the set of vertices in $H$ \textit{heavy vertices}. We call the set of vertices with $X(v)=0$ \textit{tight vertices}.

\noindent \textbf{Case 1:} Suppose that $|H|\leq 2^{3d}+2^{2d}$.

Then we can bound the total number of pebbles beyond $\frac{2^d}{d}$ by $2^{4d+1}\left(1-\frac{1}{d}\right)$. Since $P$ is not solvable and the number of pebbles at any heavy vertex is at most $2^d$ (otherwise we could pebble $r$), $X(u)<2^d-\frac{2^d}{d}=2^d\left(1-\frac{1}{d}\right)$. It follows that
%\john{Why is this needed in case 1? Isn't this only needed for Case 2?} \franklin{This is just the argument that the excess is bounded by $2^{4d+1}$ which is the excess in the formula.}
\[
\sum_{v\in V(G),X(v)>0} X(v) \leq |H| \left(2^d-\frac{2^d}{d}\right) \leq (2^{3d}+2^{2d})(2^d)\left(1-\frac{1}{d}\right) \leq 2^{4d+1}\left(1-\frac{1}{d}\right).
\]
Since there are at most $\frac{2^d}{d}-1$ pebbles on each of the non-heavy vertices, the number of pebbles in $G$ is given by
\begin{align*}
\left(\frac{2^d}{d}-1\right)|V(G)\setminus V(H)| + 2^{4d+1}\left(1-\frac{1}{d}\right) &\leq n \left(\frac{2^d}{d}-1\right) + 2^{4d+1}\left(1-\frac{1}{d}\right)< f(n,d).
\end{align*}
Thus, the theorem holds.

\noindent \textbf{Case 2:} Suppose that $|H|> 2^{3d}+2^{2d}$.

We will apply a discharging argument on the vertices of $G$. The initial charge on each vertex $v\in V(G)$ is $X(v)$. Note that we will only apply the following discharge rule a single time over all vertices of $G$. 
For each heavy vertex $v\in H$ (that is $X(v)>0$), remove charge $X(v)+1$ and distribute $X(v)+1$ uniformly over the vertices in $C_v=\{u\in V(G): u \in N^{\lceil d-\log_2(d)\rceil}(v), X(u)\neq 0\}$, 
%\john{Should the exponent on the neighborhood in $C_v$ be rounded up or down? I think it should be up but double check. Use the following statement somewhere as it seems useful: ``Since $d-\log_2(d)>d/2$ for all $d\neq 3$ and $\lceil d-\log_2(d)\rceil = \lceil d/2\rceil$ for $d=3,\ldots$''} 
that is, the set of vertices in $G$ that are distance at most $\lceil d-\log_2(d)\rceil$ away from $v$ with non-zero charge.
Note that $C_v$ may include heavy vertices. 
%\john{Isn't it possible that a heavy vertex could get a positive charge if it is within $d-\log_2(d)$ of another heavy vertex?}
The two following claims will show that each non-tight vertex receives at most $2^d\cdot \frac{1}{2^d}=1$ unit of total charge. Before discharging, the sum of the excess over all vertices of $G$ was positive. By proving these two claims, we will show that the sum of the charge on each vertex after discharging is non-positive. But since the total charge in the graph does not change and the sum of the charge on each vertex is equal to the sum of the excess over all vertices of $G$, this is a contradiction, and thus our result will hold. This will show that the number of pebbles initially on $V(G)$ is bounded above by $n\left(\frac{2^d}{d}-1\right)$.
%\john{Finish explaining why this discharging will work.}

\noindent \textbf{Claim 1:} Each vertex $v\in V(G)$ receives charge from at most $2^d$ heavy vertices.

\begin{claimproof}
Define $H_v = \{u\in V(H) \,:\, u\in N^{\lceil d-\log_2(d)\rceil}(v), X(u)\geq 1\}$. Observe that each vertex in $H_v$ can send a pebble to $v$. If $|H_v|\geq 2^d$ then $v$ would have $2^d$ pebbles and thus could pebble $r$. 
\end{claimproof}

\noindent\textbf{Claim 2:} For any vertex $v$ with $X(v)\neq 0$, the charge received from any heavy vertex $u$ in $N^{\lceil d-\log_2(d)\rceil}(v)$ is at most $\frac{1}{2^d}$. 

\begin{claimproof}
Another way to say this claim would be: for any vertex $v$ with $X(v)\neq 0$, for each heavy vertex $u$ that is distance at most $\lceil d-\log_2(d)\rceil$ from $v$, $|C_u|\geq 2^d$. In which case, $u$ will discharge $1/2^d$ to all vertices in $C_u$. 

For the remainder of this claim, let $v\in V(G)$ be some arbitrary vertex in $V(G)$. Let $\tau$ be a spanning directed BFS tree rooted at $v$ so that there is a directed path in $\tau$ to $v$ from all vertices in $V(G)$. 
Note that the largest distance between any vertex in $\tau$ and $v$ is $d$ since we built a BFS tree.
We define the ancestor of a vertex $u$ in $\tau$ as any vertex $x$ for which there is a directed path from $u$ to $x$ and a descendant of a vertex $u$ in $\tau$ as any vertex $x$ for which there is a directed path from $x$ to $u$.
Define $A_v=\{u\in V(G)\,:\, u\in \partial( N^{\lceil d-\log_2(d)\rceil}(v),X(u)\neq 0, \text{ and $u$ is an ancestor of some $w\in H$ in $\tau$}\}$, that is, $A_v$ is the set of vertices $u$ in $V(G)$ distance exactly $\lceil d-\log_2(d)\rceil$ away from $v$ with non-zero charge such that there is a directed path in $\tau$ from some heavy vertex to $u$. 
Notice that $A_v\subseteq C_v$ by definition of $A_v$ and $C_v$. Therefore, we need only show that $|A_v|\geq 2^d$ to show that $|C_v| \geq 2^d$. 

By Claim 1, there are at most $2^d$ heavy vertices in $N^{\lceil d-\log_2(d)\rceil }(v)$. The number of tight vertices in $\partial( N^{\lceil d-\log_2(d)\rceil}(v))$ that have a heavy vertex descendant in $\tau$ is at most $2^d-1$ since these tight vertices can be made heavy vertices by pebbling the tight vertices via the directed path from a heavy vertex to a tight vertex. 
Each tight vertex in $\partial(N^{\lceil d-\log_2(d)\rceil}(v))$ can have at most $2d^2$ heavy descendants in $\tau$. 
% \john{More justification must be used for the previous result. I nominate Franklin for this.}
Thus there are at most $2d^2(2^d-1)$ heavy vertices that have tight vertex ancestors in $\partial( N^{\lceil d-\log_2(d)\rceil}(v))$. 

So far, we have determined the position of at most $2^d+2d^2(2^d-1)=2^d(2d^2+1)+1$ vertices in $H$. Since $|H|> 2^{3d}+2^{2d}$, there are at least $2^{3d}$ heavy vertices left to identify. Since there are no cycles in the BFS tree, all of the unidentified vertices in $H$ must have a unique ancestor in $\tau$ from $A_v$, thus the number of unidentified vertices in $H$ will help us find the size of $A_v$. Again by Claim 1, each vertex in $A_v$ is the ancestor in $\tau$ of at most $2^d$ vertices in $H$. Therefore, by our assumption on the size of $H$, 
\[
|A_v|\geq \frac{|H|-2^d(2d^2+1)-1}{2^d} \geq 2^{2d}+2^d-(2d^2+1)-\frac{1}{2^d}\geq 2^{2d}.
\]
%\john{Why don't we simply prove that $|A_v|\geq 2^d$? Wouldn't this be enough to prove the claim as well?}
Since $v$ was arbitrarily chosen, this claim holds for all $v\in V(G)$, which suffices to prove the claim. 
\end{claimproof}

Thus $f(n,d)\leq n\left(\frac{2^d}{d}-1\right)+2^{4d+1}\left(1-\frac{1}{d}\right)$, as expected.
\end{proof}

\bibliographystyle{amsplain}

\end{document}